\documentclass[12pt]{amsart}
\usepackage{amssymb}
\usepackage{hyperref}

\theoremstyle{plain}
\newtheorem{theorem}{Theorem}[section]
\newtheorem{proposition}[theorem]{Proposition}

\newtheorem{lemma}[theorem]{Lemma}

\theoremstyle{definition}
\newtheorem{definition}[theorem]{Definition}
\newtheorem{example}[theorem]{Example}
\newtheorem{remark}[theorem]{Remark}
\newtheorem{question}[theorem]{Question}
\newtheorem{conjecture}[theorem]{Conjecture}

\theoremstyle{remark}

\numberwithin{equation}{section}

\newcommand{\Z}{\mathbb Z}

\newcommand{\C}{\mathbb C}

\newcommand{\SO}{\operatorname{SO}}

\newcommand{\op}{\operatorname}

\newcommand{\spec}{\operatorname{Spec}}

\newcommand{\diag}{\operatorname{diag}}

\newcommand{\ba}{\backslash}
\newcommand{\mult}{\operatorname{mult}}

\newcommand{\norma}[1]{\|{#1}\|_1}

\title[A computational study on lens spaces]{A computational study on lens spaces isospectral on forms}
\author{Emilio~A.~Lauret}
\address{CIEM--FaMAF (CONICET), Universidad Nacional de C\'ordoba, Medina Allende s/n, Ciudad Universitaria, 5000 C\'ordoba, Argentina.}
\email{elauret@famaf.unc.edu.ar}
\date{March 2017}

\begin{document}

\begin{abstract}
We make a computational study to know what kind of isospectralities among lens spaces and lens orbifolds exist considering the Hodge--Laplace operators acting on smooth $p$-forms.
Several facts evidenced by the computational results are proved and some others are conjectured.
\end{abstract}

\maketitle

\tableofcontents

\section{Introduction}\label{sec:intro}

Let $M$ be a compact Riemannian manifold of dimension $d$ without boundary.
For $0\leq p\leq d$, we will denote by $\Delta_p$ the Hodge-Laplace operator acting on smooth $p$-forms of $M$.
It is well known that the spectrum of $\Delta_p$, denoted by $\spec_p(M)$, is a discrete set of non-negative real numbers, repeated according to its finite multiplicity, and tending to infinity.
If $\spec_p(M)=\spec_p(M')$, $M$ and $M'$ are said to be \emph{$p$-isospectral}.

The following question appears naturally.

\begin{question}\label{question1}
For a given subset $I$ of $\{0,\dots,d\}$.
Are there $d$-dimensional non-isometric compact Riemannian manifolds $M$ and $M'$ such that they are $p$-isospectral if and only if $p$ is in $I$?
\end{question}

The main goal of this article is to study Question~\ref{question1}, for every choice $I\subset \{0,\dots,d\}$, in the class of lens spaces, and more generally in the class of odd-dimensional lens orbifolds.

A lens space is an orientable manifold with positive constant sectional curvature and cyclic fundamental group.
Its dimension is odd and it has the form $\Gamma\ba S^{2n-1}$ with $\Gamma$ a cyclic subgroup of $\SO(2n)$ acting freely on $S^{2n-1}$.
Relaxing the free action condition of $\Gamma$ we obtain a lens orbifold.
We will always assume that the lens orbifolds are odd-dimensional.

Lens spaces has been used many times as a test case for spectral questions, since their spectra can be explicitly computed.
Ikeda used generating functions to encode the $p$-spectra of a spherical space form.
This idea was very useful to construct various isospectral examples (\cite{Ikeda80_isosp-lens}, \cite{Ikeda83}, \cite{Ikeda88}, \cite{GornetMcGowan06}) and also to prove spectral rigidity results (\cite{IkedaYamamoto79}, \cite{Yamamoto80}, \cite{Ikeda80_3-dimI}, \cite{Ikeda80_3-dimII}, \cite{Ikeda97}).
Among many other results, Ikeda showed in \cite{Ikeda88} for each $p_0\geq0$ a pair of lens spaces which are $p$-isospectral for all $p$ satisfying $0\leq p\leq p_0$, and are not $(p_0+1)$-isospectral.
Subsections~\ref{subsec:Ikeda'sapproach} and \ref{subsec:Ikeda'sexamples} give a summary on some of these results.

Gornet and McGowan~\cite{GornetMcGowan06} reactivated the use of lens spaces in spectral questions by making a computational study of $p$-isospectral lens spaces (see Subsection~\ref{subsec:GornetMcGowan} for details).
They used Ikeda's generating functions to check whether two lens spaces are $p$-isospectral.
Similarly, Shams Ul Bari~\cite{Shams11} found several examples of $0$-isospectral lens orbifolds. 
Also following Ikeda's approach, the Dirac operator was also considered on spin lens spaces in \cite{Bar91thesis} and \cite{Boldt15}, and in spherical space forms in \cite{Bar96}.
Very recently, Bari and Hunsicker~\cite{ShamsHunsicker17} proved the non-existence for non-isometric $0$-isospectral lens orbifolds in dimension $3$ and $4$.

By using standard representation theory on compact groups, Miatello, Rossetti and the author~\cite{LMR-onenorm} relate the $p$-spectrum of a lens space with the number of vectors of fixed one-norm in certain associated sublattice of $\Z^n$.
This was used to show the first examples of compact Riemannian manifolds that are $p$-isospectral for all $p$ but are not strongly isospectral.
The articles \cite{LMR-survey}, \cite{BoldtLauret-onenormDirac}, \cite{Lauret-spec0cyclic} and \cite{Lauret-multip} follow this approach and \cite{DeFordDoyle14} study in detail the examples of all-$p$-isospectral pairs in \cite{LMR-onenorm}.
The articles \cite{MohadesHonori16}, \cite{MohadesHonori16b} are also related to this approach.

In order to explain in detail the main goals of this article, we introduce some useful notation.

\begin{definition}\label{notation:I-isospectral}
For $I$ a subset of $\{0,\dots,d\}$, we will say that a family (two or more) of $d$-dimensional manifolds are \emph{$I$-isospectral} if they are mutually $p$-isospectral for all $p\in I$, and for any $p\not\in I$ there are at least two elements in the family that are not $p$-isospectral.
\end{definition}

In other words, $I$ is the largest subset of $\{0,\dots,d\}$ such that any two elements in the family are $p$-isospectral for all $p$ in it.
In particular, any subfamily of an $I$-isospectral family will be $I'$-isospectral for some $I'\supseteq I$.

For an orientable $d$-dimensional compact Riemannian manifold $M$, $\spec_p(M)=\spec_{d-p}(M)$ for all $p$.
Hence, when the underlying manifolds are orientable, we abbreviate $I$-isospectral for some $I\subset \{0,1,\dots, \lfloor \frac{d}{2} \rfloor\}$ to $I'$-isospectral when $I'=I\cup \{d-p:p\in I\}$.
Lens orbifolds are orientable of odd dimension, say $d=2n-1$, thus we will be considering subsets of $\{0,\dots,n-1\}$.

The aim of this article is to make a computational study of $I$-isospectral lens orbifolds.
The appendix \cite{Lauret-appendix} includes, for low values of $n$ and $q$, all families of $(2n-1)$-dimensional $I$-isospectral lens orbifolds with fundamental group of order $q$.
Section~\ref{sec:data} includes summaries of these computational results, by showing the subsets $I$ of $\{0,\dots,n-1\}$ for which there exists an $I$-isospectral family.
Section~\ref{sec:evidencedfacts} proves several facts evidenced from the data by using the tools introduced in Section~\ref{sec:preliminaries}.
All computations were made by using \cite{Sage}.

We end this section by listing the most interesting conclusions.

\begin{enumerate}
\item The most common obstruction to the existence of $I$-isospectral families is the `hole obstruction' in Proposition~\ref{prop:hole-obstruction}. If two spherical orbifolds (e.g.\ lens orbifolds) are $(p-1)$-isospectral and $(p+1)$-isospectral, then they also are $p$-isospectral. In other words, the subset $I$ cannot contain a `simple hole'.

\item It is well known that a lens space cannot be $0$-isospectral to a lens orbifold with singularities since they share a common Riemannian cover (see Subsection~\ref{subsec:Orbifoldsmanifolds}).
    The computational results give strong evidences that a lens space cannot be $p$-isospectral to a lens orbifold with singularities for any $p$ (see Conjecture~\ref{conj:orbifold/manifold}).

\item We prove the non-existence of a pair of $I$-isospectral lens spaces for some choices of $I$.
    Namely, when $I\subset \{0,\dots,n-1\}$ has $n-1$ elements (see Theorem~\ref{thm:[0,n-2]-isosp}), and when $I$ has $n-2$ elements and it is different to $\{0,\dots,n-3\}$ (see Theorem~\ref{thm:n-2elements}).

\item Gornet and McGowan in \cite{GornetMcGowan06} were interested in the existence of pairs of $p$-isospectral lens spaces for some $p>0$, which are not $0$-isospectral.
    We found several such examples from dimension $5$ on among lens orbifolds.
    We also found such examples among lens spaces from dimension $11$ on (see Remark~\ref{rem:p-isospp>0}).

\item There is evidence that the isotropy type of the singular points of $0$-isospectral lens orbifolds coincide (see Remark~\ref{rem:L(q/d;s)orbifolds}).
    Furthermore, one can see that most $I$-isospectral families with $I$ non-empty and $0\not\in I$ have lens orbifolds with singular points of different isotropy types (see Remark~\ref{rem:isotropytype}).

\item If two lens spaces are $0$-isospectral, then the pair of corresponding covering spaces of the same degree are also $0$-isospectral (see Subsection~\ref{subsec:increasing-q}).
    This is not true for $p>0$.
\end{enumerate}

\section{Preliminaries}\label{sec:preliminaries}

In this section we will introduce several concepts and results on lens orbifolds and their $p$-spectra.
It is based on the references \cite{IkedaTaniguchi78}, \cite{IkedaYamamoto79}, \cite{Ikeda80_isosp-lens}, \cite{Ikeda88}, \cite{LMR-repequiv}, \cite{LMR-onenorm}, \cite{Lauret-spec0cyclic}, \cite{Lauret-multip}.
These preliminaries will be used in Section~\ref{sec:evidencedfacts} to prove several facts evidenced by the computational results.
The author suggests the reader already familiar with lens spaces and their $p$-spectra to skip this section and come back to it for reference.

\subsection{$p$-Spectra of spherical orbifolds}

We assume that $M$ is a \emph{good orbifold} covered by $S^{2n-1}$, that is, $M=\Gamma\ba S^{2n-1}$ with $\Gamma$ a finite subgroup of $\SO(2n)$.
We will use the term \emph{spherical orbifold} for a space as above.
The space $\Gamma\ba S^{2n-1}$ is a manifold if and only if $\Gamma$ acts freely on $S^{2n-1}$.
In this case, $\Gamma\ba S^{2n-1}$ is usually called a \emph{spherical space form}.

For $\Gamma\subset G$ finite, let us denote by $\Delta_{\Gamma,p}$ the Hodge-Laplace operator on $p$-forms of $\Gamma\ba S^{2n-1}$, which is given by the restriction of $\Delta_p$ on $S^{2n-1}$ to $\Gamma$-invariant smooth $p$-forms on $S^{2n-1}$.
Two spherical orbifolds $\Gamma\ba S^{2n-1}$ and $\Gamma'\ba S^{2n-1}$ are said to be \emph{$p$-isospectral} if the operators $\Delta_{\Gamma,p}$ and $\Delta_{\Gamma',p}$ have the same spectra.
The space $\Gamma\ba S^{2n-1}$ is orientable, thus $\spec_p(\Gamma\ba S^{2n-1}) = \spec_{2n-1-p}(\Gamma\ba S^{2n-1})$.
In particular, $p$-isospectrality for every $0\leq p\leq n-1$ is actually equivalent to $p$-isospectrality for every $p$.

The next theorem is well known (see for instance \cite[Thm.~4.2]{IkedaTaniguchi78}, \cite[Prop.~2.1]{Ikeda88}, \cite[Thm.~1.1]{LMR-repequiv}, \cite[Prop.~2.2]{LMR-onenorm}).
It describes the $p$-spectrum of $\Gamma\ba S^{2n-1}$ in terms of the dimension of the $\Gamma$-invariant subspaces of certain irreducible representations of $\SO(2n)$.
For $\pi$ a representation of $\SO(2n)$, let $V_\pi$ denote its underlying vector space and $V_{\pi}^\Gamma$ the $\Gamma$-invariants in $V_{\pi}$.
Following the same notation as in \cite{Lauret-multip}, let $\pi_{k,p}$ be the irreducible representation of $\SO(2n)$ with highest weight $k\varepsilon_1+\sum_{j=1}^p \varepsilon_j = (k+1)\varepsilon_1+\sum_{j=2}^p\varepsilon_j$.
For $k\geq 1$, we set
\begin{equation}\label{eq1:lambda_(k,p)}
\lambda_{k,p} =
\begin{cases}
0&\quad\text{if $p=-1$},\\
(k+p)(k+2n-2-p) &\quad\text{if $0\leq p\leq n-1$.}
\end{cases}
\end{equation}

\begin{theorem}\label{thm:spectrum-general}
Fix $0\leq p\leq n-1$.
Each eigenvalue in $\spec_p(\Gamma\ba S^{2n-1})$ is of the form  $\lambda_{k,p-1}$ or $\lambda_{k,p}$ for some $k\geq 1$, with multiplicity
\begin{equation*}
\mult_{\Delta_{\Gamma,p}}(\lambda_{k,p-1}) = \dim V_{\pi_{k-1,p}}^\Gamma
\quad\text{ and }\quad
\mult_{\Delta_{\Gamma,p}}(\lambda_{k,p}) = \dim V_{\pi_{k-1,p+1}}^\Gamma
\end{equation*}
respectively.
\end{theorem}

One can see from Theorem~\ref{thm:spectrum-general} that the $p$-spectrum of a spherical orbifold consists of two `strings' of multiplicities, namely $\{\mult_{\Delta_{\Gamma,p}}(\lambda_{k,p-1})\}_{k\geq1}$ and $\{\mult_{\Delta_{\Gamma,p}}(\lambda_{k,p})\}_{k\geq1}$.
Consequently, the $p$-spectrum and the $(p+1)$-spectrum of a spherical orbifold share one string.
This remark yields a very important obstruction to isospectrality (see for instance \cite[Cor.~1.2(ii)]{LMR-repequiv} or \cite[Prop.~2.1]{Ikeda88}).

\begin{proposition}\label{prop:hole-obstruction}
If two $(2n-1)$-dimensional spherical orbifolds are $(p-1)$-isospectral and $(p+1)$-isospectral, then they are $p$-isospectral.
In particular, $1$-isospectrality implies $0$-isospectrality.
\end{proposition}

\subsection{Lens spaces}\label{subsec:lensspaces}
Let $q$ be a positive integer and let $s=(s_1,\dots,s_n)\in\Z^n$ satisfying that $\gcd(q,s_1,\dots,s_n)=1$.
The cyclic group $\Gamma_{q,s}$ generated by
\begin{equation*}\label{eq:gamma}
\gamma_{q,s}:=
\diag\left(
\left[\begin{smallmatrix}\cos(2\pi{s_1}/q)&-\sin(2\pi{s_1}/q) \\ \sin(2\pi{s_1}/q)&\cos(2\pi{s_1}/q)
\end{smallmatrix}\right]
,\dots,
\left[\begin{smallmatrix}\cos(2\pi{s_m}/q)&-\sin(2\pi{s_m}/q) \\ \sin(2\pi{s_m}/q)&\cos(2\pi{s_m}/q)
\end{smallmatrix}\right]
\right)
\end{equation*}
has order $q$ and the space
\begin{equation*}
L(q;s)=L(q;s_1,\dots,s_m):=\Gamma_{q,s}\ba S^{2n-1},
\end{equation*}
is a \emph{lens orbifold}.
We will always consider the Riemannian structure induced by the round metric on $S^{2n-1}$.
The group $\Gamma_{q,s}$ acts freely on $S^{2n-1}$ if and only if $\gcd(q,s_j)=1$ for all $j$.
In this case, the Riemannian manifold $\Gamma_{q,s}\ba S^{2n-1}$ is a \emph{lens space}.
The following fact is well known (see for instance \cite[Ch.~V]{Cohen-book}).

\begin{proposition}\label{prop3:lens-isom}
Let $L=L(q;s)$ and $L'=L(q;s')$ be orbifold lens spaces.
Then, $L$ and $L'$ are isometric if and only if there exist a permutation $\sigma$ of $\{1,\dots,n\}$, $\epsilon_1,\dots,\epsilon_n\in\{\pm1\}$ and $t\in\Z$ coprime to $q$, such that
  $$
  s_{\sigma(j)}'\equiv t\epsilon_js_j\pmod{q}
  $$
  for all $1\leq j\leq n$.
\end{proposition}

\begin{remark}\label{rem:lensinC^n}
An equivalent and useful way to view the lens space $L(q;s)$ is as the quotient of $S^{2n-1}=\{(z_1,\dots,z_n)\in\C^n:|z_1|^2+\dots+|z_n|^2\}\subset \C^n$ by the action of the cyclic group of $q$-roots of unity given by
\begin{equation*}
  \xi\cdot (z_1,\dots,z_n)=(\xi^{s_1}z_1,\dots,\xi^{s_n}z_n).
\end{equation*}
\end{remark}

The isotropy group of a point $x$ in a lens orbifold $L(q;s)=S^{2n-1}/\Gamma_{q,s}$, which in particular is a good orbifold, is given by the elements in $\Gamma_{q,s}$ fixing $x$.
Two points share the isotropy type if their isotropy groups are isomorphic.
The connected components of the equivalent isotropy classes of the points in $L(q;s)$ form a stratification of $L(q;s)$.
The points with a non-trivial isotropy group are called \emph{singular}, and otherwise \emph{regular}.
Since $L(q;s)$ is connected, the subset of its regular points form a single stratum, the only stratum of full dimension.
More details on orbifolds and their spectra can be found in \cite{Gordon12-orbifold}.

The isotropy classes in a lens orbifold $L(q;s)$ is determined by the multiset (set with multiplicities)
\begin{equation}\label{eq:multisetGCD}
\{\!\{\gcd(q,s_j):1\leq j\leq n\}\!\}.
\end{equation}
For example, $L(q;s)$ is a lens space if and only if \eqref{eq:multisetGCD} is equal to $\{\!\{1^n\}\!\}$ (the multiset given by $1$ repeated $n$ times).
The lens orbifold $L(4;1,0)$ contains exactly one point with non-trivial isotropy group.
Indeed, the action of a $q$-root of unity $\xi$ on a point $(z_1,z_2)$ in $S^3$ is
$$
\xi\cdot(z_1,z_2)= (\xi z_1,z_2),
$$
thus the class of the point $(0,1)$ in $L(4;1,0)$ has singular isotropy $\Gamma_{4,(1,0)}$, while $\Gamma_{4,(1,0)}$ acts freely on any other point in $L(4;1,0)$.

A more involved example is $L:=L(4;0,1,2)$, which has the three different isotropy types.
Indeed, the class in $L$ of the point $(1,0,0)$ in $S^5$ has isotropy group equal to $\Gamma_{4,(0,1,2)}$, and the class in $L$ of the points $(a,0,b)$ in $S^5$ satisfying $a^2+b^2=1$ and $b\neq 0$ have isotropy group of order $2$.
The rest of the points are regular.

\subsection{Ikeda's approach}\label{subsec:Ikeda'sapproach}
A.~Ikeda made important progress in inverse spectral geometry of spherical space forms, particularly in the study of lens spaces.
His main tool were the (two) generating functions associated to the $p$-spectrum of a lens space given.
He defined (see \cite[(2.3)]{Ikeda88})
\begin{equation}\label{eq1:F_Gamma^p}
F_\Gamma^{p}(z) =
\sum_{k\geq0}  \mult_{\Delta_{\Gamma,p}} (\lambda_{k+1,p}) z^{k}
=\sum_{k\geq0} \dim V_{\pi_{k,p+1}}^\Gamma z^k,
\end{equation}
for any $0\leq p\leq n-1$ and any finite subgroup $\Gamma$ of $\SO(2n)$.
By Theorem~\ref{thm:spectrum-general}, $F_\Gamma^{p-1}(z)$ and $F_\Gamma^{p}(z)$ encode the $p$-spectrum of $\Gamma\ba S^{2n-1}$, and each generating function corresponds to each string.
From Theorem~\ref{thm:spectrum-general}, one immediately obtains the following characterization of $p$-isospectral spherical orbifolds.

\begin{proposition}\label{prop:p-isospIkeda}
Two spherical orbifolds $\Gamma\ba S^{2n-1}$ and $\Gamma'\ba S^{2n-1}$ are $p$-isospectral if and only if
$F_{\Gamma}^{p-1}(z)=F_{\Gamma'}^{p-1}(z)$ and $F_{\Gamma}^{p}(z)=F_{\Gamma'}^{p}(z).
$
\end{proposition}

We next recall the expression for $F_\Gamma^p(z)$ obtained by Ikeda.
In \cite[Thm.~3.2]{IkedaYamamoto79}, for $\Gamma$ a finite subgroup of $\SO(2n)$, it was shown that
\begin{equation}\label{eq:F_Gamma^0}
F_\Gamma^0(z) = \frac{1}{z}\left(\frac{1-z^2}{q}\sum_{\gamma\in\Gamma} \frac{1}{ \det(z-\gamma )}-1\right).
\end{equation}
Here $\det(z-\gamma)$ stands for $\det(\op{Id}_{2m}z-\gamma)=\prod_\lambda (z-\lambda)$, where $\lambda$ runs over the eigenvalues of $\gamma$.
In the case when $\Gamma\ba S^{2n-1}=L(q;s_1,\dots,s_n)$, we have that
\begin{equation}\label{eq:F_L^0}
F_\Gamma^0(z) = \frac{1}{z}\left(\frac{1-z^2}{q}\sum_{h=0}^{q-1} \prod_{j=1}^n \frac{1}{ (z-\xi_q^{h s_j})(z-\xi_q^{-h s_j})}-1\right).
\end{equation}
This equation was used in \cite{Ikeda80_isosp-lens} to give the first pair of $0$-isospectral lens spaces.

\begin{remark}\label{rem:F_L^0}
We point out that \eqref{eq:F_Gamma^0} is not identical to the one given in \cite{IkedaYamamoto79} and \cite{Ikeda80_isosp-lens}, since there $F_\Gamma^0(z)$ was defined by $\sum_{k\geq0} \dim V_{\pi_{k-1,1}}^\Gamma z^k$.
Actually, \cite[(4.7)]{LMR-survey} is missing the term $z^{-1}$.

From now on, for $L=\Gamma\ba S^{2n-1}$ a lens orbifold, we will write $F_L^p(z)$ in place of $F_\Gamma^p(z)$.
Furthermore, for a lens space $L=L(q;s)$, we write
\begin{equation}\label{eq:F_L^0abreviado}
\widetilde F_{L}^0(z) = \frac{q(zF_L^0(z)+1)}{(1-z^2) }
= \sum_{h=0}^{q-1} \prod_{j=1}^n \frac{1}{ (z-\xi_q^{h s_j})(z-\xi_q^{-h s_j})}.
\end{equation}
This new function will be very useful since $L$ and $L'$ are $0$-isospectral if and only if $\widetilde F_{L}^0(z)=\widetilde F_{L'}^0(z)$.
\end{remark}

Furthermore, Ikeda also gave an expression for $F_{\Gamma}^p(z)$ for $1\leq p\leq n-1$ and $\Gamma$ a finite subset of $\SO(2n)$, namely (see \cite[p.~394]{Ikeda88}),
\begin{equation}\label{eq:neat}
F_\Gamma^p(z) = (-1)^{p+1} z^{-p} + \frac{1}{|\Gamma|} \sum_{k=0}^p (-1)^{p-k} (z^{k-p}-z^{p-k+2})\sum_{\gamma\in\Gamma} \frac{\chi^{k}(\gamma)}{\det(z-\gamma)}.
\end{equation}
Here, $\chi^k$ denotes the character of the $k$-exterior representation $\bigwedge^k(\C^{2m})$ of $\SO(2m)$.

\subsection{Ikeda's examples}\label{subsec:Ikeda'sexamples}
We now recall briefly the examples of families of lens spaces isospectral for $p$ satisfying $0\leq p\leq p_0$.
They were given by Ikeda in \cite{Ikeda88} (see also \cite[\S4]{LMR-survey}).
These examples will be cited several times in the rest of the article.

For $q$ and $n$ positive integers, let ${\mathfrak L}_0(q;n)$ denote the classes (up to isometry) of $(2n-1)$-dimensional lens spaces $L(q;s_1,\dots,s_n)$ satisfying $s_i\not\equiv \pm s_j\pmod q$ for all $i\neq j$.
For a lens space $L=L(q;s_1,\dots,s_n)$ in ${\mathfrak L}_0(q;n)$, choose $h$ integers $\bar s_1,\dots,\bar s_h$ such that
$$
\{\pm s_1,\dots, \pm s_n, \pm \bar s_1,\dots,\pm\bar s_h\}
$$
is a set of representatives of integers mod~$q$, coprime to $q$.
Therefore $2n + 2h = \phi(q)$, where $\phi(q)$ denotes the Euler phi function.
Set $\bar L=L(q;\bar s_1,\dots,\bar s_h)$.
One can easily check that two lens spaces $L$ and $L'$ in ${\mathfrak L}_0(q;n)$ are isometric if and only if $\overline{L}$ and $\overline{L'}$ are isometric in ${\mathfrak L}_0(q;h)$ (see \cite[Prop.~3.3]{Ikeda88}).

We now assume that $q$ is a prime number and $n$ satisfies $q-1=2n+4$, in order to have $h=2$.
Set
\begin{equation}
{\mathfrak L}_p(q;n) = \left\{L(q;s)\in {\mathfrak L}_0(q;n):
\begin{array}{l}
a_1\bar s_1+a_2\bar s_2\not\equiv0\pmod q ,\\
\forall \; 1\leq |a_1|+|a_2|\leq p+2
\end{array}
\right\},
\end{equation}
thus one has the filtration
\begin{equation}\label{eq:filtration}
{\mathfrak L}_0(q;n) \supset {\mathfrak L}_1(q;n) \supset {\mathfrak L}_2(q;n)\supset\cdots\,.
\end{equation}
With this notation we can state Ikeda's result.

\begin{theorem}\label{thm:mainIkeda}
Let $q$ be an odd prime, $n=(q-5)/2$ and $0\leq p_0\leq n-1$ and let $L$ and $L'$ be lens spaces in ${\mathfrak L}_{p_0}(q;n)$.
Then $L$ and $L'$ are $p$-isospectral for all $0\leq p\leq p_0$.
If furthermore $L\in {\mathfrak L}_{p_0+1}(q;n)$ and $L'\notin {\mathfrak L}_{p_0+1}(q;n)$, then $L$ and $L'$ are not $(p_0+1)$-isospectral.
\end{theorem}

\begin{example}\label{ex:Ikeda'sexamples}
We now give the first steps of this sequence of families.
The cases satisfying $q<11$ are not interesting since ${\mathfrak L}_{0}(q;(q-5)/2)$ has less than two elements.
If $q=11$, then $n=3$ and we have that $L_1:=L(11;3,4,5)\simeq L(11;1,2,4) \in {\mathfrak L}_{0}(11;3) \smallsetminus {\mathfrak L}_{1}(11;3)$ and $L_2:=L(11;2,4,5)\simeq L(11;1,2,3) \in {\mathfrak L}_{1}(11;3)$.
Hence, $L_1$ and $L_2$ are $0$-isospectral but not $1$-isospectral, that is, they are $\{0\}$-isospectral according to Notation~\ref{notation:I-isospectral}.

For $q=13$ we have that $n=4$ and one can check that ${\mathfrak L}_{0}(13;4)$ is given by the elements
$$
\begin{array}{l}
L_1:=L(11;3,4,5,6)\simeq L(11;1,2,3,6) \in {\mathfrak L}_{0}(13;4) \smallsetminus {\mathfrak L}_{1}(13;4), \\
L_2:=L(11;2,4,5,6)\simeq L(11;1,2,3,4) \in {\mathfrak L}_{1}(13;4)\smallsetminus {\mathfrak L}_{2}(13;4),\\
L_3:=L(11;2,3,4,6)\simeq L(11;1,2,3,5) \in {\mathfrak L}_{2}(13;4).
\end{array}
$$
Therefore, these manifolds are all $0$-isospectral pairwise, $L_2$ and $L_3$ are also $1$-isospectral, $L_1$ is not $1$-isospectral to any other in the family, and there are no $2$-isospectral pairs among them.
In other words, the family $\{L_1,L_2,L_3\}$ is $\{0\}$-isospectral, $\{L_2,L_3\}$ is $\{0,1\}$-isospectral and there is no any $\{0,1,2\}$-isospectral subfamily of $\{L_1,L_2,L_3\}$. 
\end{example}

\begin{remark}
	Adapting this method, Shams~\cite{Shams11} found many families of $0$-isospectral lens orbifolds with singular points of dimension $\geq9$. 
	He worked with $q$ a power of a prime number and also a product of two different primes numbers. 
\end{remark}

\subsection{Gornet and McGowan's computational study}\label{subsec:GornetMcGowan}
Gornet and McGowan made a computational study on $p$-isospectral lens spaces by following Ikeda's approach.
We next try to explain what they did in \cite{GornetMcGowan06}.

Fix $q$ a prime number and write $q-1=2n+2h$ with $n$ and $h$ positive integer numbers (they considered the cases $2\leq h\leq 4$).
For $L$ in $\mathfrak L_0(q;n)$, the generating function $F_L^p(z)$ in \eqref{eq:neat} can be written as (a common term plus) a polynomial over the $q$-th cyclotomic polynomial $\Phi_q(z)$ (see \cite[(4.6)]{Ikeda88} and \cite[(4.11)]{LMR-survey}).
Such polynomial is \cite[(3)]{GornetMcGowan06}, and its coefficients depend on certain arithmetic conditions (see \cite[(5)]{GornetMcGowan06}).

Gornet and McGowan computed these coefficients, for every prime number $q\leq 100$ and $h=2,3,4$, and then by comparing these numbers obtained the families of lens spaces in $\mathfrak L_0(q;n)$ that are $p$-isospectral, for any $0\leq p\leq n-1$.
We note that their examples are quite different from the ones given in this paper.
First at all, they work in lens spaces contained in $\mathfrak L_0(q;n)$, while we work with arbitrary lens orbifolds.
Furthermore, the dimension of their examples is subject to the choices of $q$ and $h$ ($2n-1=q-1-2h$) and consequently it grows quickly.
In our case, we fix $q$ and the dimension $2n-1$ up to $11$ ($17$ in \cite{Lauret-appendix}).

Unfortunately, the condition of $q$ being a prime number is essential in the previous argument, and it was not assumed in \cite{GornetMcGowan06}.
Sebastian Boldt communicated to the authors the following counterexample: the $3$-dimensional lens spaces $L(15;1,2)$ and $L(15;1,4)$ with fundamental group of order $15$ are not isometric and satisfy the condition for being $0$-isospectral given in \cite{GornetMcGowan06}, which is impossible by \cite{IkedaYamamoto79}.
The author thanks Sebastian Boldt for sharing this smart example.
Ruth Gornet communicated the author in the ``VI Workshop on Differential Geometry 2016'' at C\'ordoba, Argentina, that she and Jeffrey McGowan are working to extend their results to composite numbers given by the product of two primes numbers, including the case of the square of a prime number.

\subsection{One-norm approach}

In \cite{LMR-onenorm} was started a study of the spectra of lens spaces in connection with the one-norm lengths of elements in the associated \emph{congruence lattice}.
Apparently, this relation was already known by Ikeda and Yamamoto as it is indicated in \cite{Yamamoto80}.
However, they did not make any further use of this connection.
This approach works also for spherical orbifolds of the form $\Gamma\ba S^{2n-1}$ with $\Gamma$ a finite abelian subgroup of $\SO(2n)$.

We associate to a lens orbifold $L(q;s_1,\dots,s_n)$ the \emph{congruence lattice}
\begin{equation}\label{eq1:conglattice}
\mathcal L(q;s_1,\dots,s_n) :=
\left\{\sum_{j=1}^n a_j\varepsilon_j \in \Z^n: a_1s_1+\dots+a_ns_n\equiv0\pmod q\right\}.
\end{equation}
For $\mu=(a_1,\dots,a_n)\in\Z^n$, we write $\norma{\mu}=|a_1|+\dots+|a_n|$ for the \emph{one-norm} of $\mu$ and let $Z(\mu)$ be the number of zero coordinates of $\mu$.
For any subset $\mathcal L$ of $\Z^n$, set
\begin{align}\label{eq1:N_L(k)N_L(k,l)}
N_{\mathcal L} (k,\ell) &= \#\{\mu\in \mathcal L: \norma{\mu}=k,\;Z(\mu)=\ell\},\\
N_{\mathcal L} (k)& = \#\{\mu\in \mathcal L: \norma{\mu}=k\} . \notag
\end{align}
Furthermore, the \emph{one-norm} generating function of $\mathcal L$ is given by
\begin{equation}
  \vartheta_{\mathcal L}(z) = \sum_{k\geq0} N_{\mathcal L}(k)z^k.
\end{equation}
Similarly, for $0\leq \ell\leq n$, we let
\begin{equation}
  \vartheta_{\mathcal L}^{(\ell)}(z) = \sum_{k\geq0} N_{\mathcal L}(k,\ell)z^k.
\end{equation}

After quite some effort, the author proved in \cite[Thm.~2.2]{Lauret-multip} an explicit expression for $F_{L}^p(z)$ for any $0\leq p\leq n-1$.

\begin{theorem}\label{thm:F_L^p}
Let $L$ be a $(2n-1)$-dimensional lens orbifold with associated congruence lattice $\mathcal L$.
For each $1\leq p\leq n$,
we have that
\begin{equation}\label{eq1:F_Gamma^p-1-formula}
F_L^{p-1}(z) = \frac{1}{z^p(1-z^2)^{n-1}}\sum_{\ell=0}^{n} \vartheta_{\mathcal L}^{(\ell)}(z) \; A_{p}^{(\ell)}(z)+\frac{(-1)^p}{z^p},
\end{equation}
where \begin{align}\label{eq3:A_pl}
A_{p}^{(\ell)}(z) &= \sum_{j=1}^{p} (-1)^{j-1}  \sum_{t=0}^{\lfloor\frac{p-j}{2}\rfloor} \binom{n-p+j+2t}{t}
\\ &\quad
\sum_{\beta=0}^{p-j-2t} 2^{p-j-2t-\beta} \binom{n-\ell}{\beta} \binom{\ell}{p-j-2t-\beta}
\sum_{\alpha=0}^\beta \binom{\beta}{\alpha} \sum_{i=0}^{j-1} z^{2(p-j-t-\alpha+i)}. \notag
\end{align}
\end{theorem}

\begin{remark}
It is important to note that $A_{p}^{(\ell)}(z)$ does not depend on the particular lens orbifold $L$.
Furthermore, it is a polynomial of degree $\leq 2(p-1)$ since $0\leq 2(p-j-2t-\beta) \leq 2(p-j-t-\alpha+i)\leq 2(p-j+j-1)=2(p-1)$.
\end{remark}

For example, we have that
\begin{align*}
F_L^0(z)
&= -\frac{1}{z}+\frac{1}{z(1-z^2)^{n-1}} \sum_{\ell=0}^n \theta_{\mathcal L}^{(\ell)}(z), \\
F_L^1(z)
&= \frac{1}{z^2}+ \frac{1}{z^2(1-z^2)^{n-1}} \sum_{\ell=0}^{n} \theta_{\mathcal L}^{(\ell)}(z) \Big((n-1+\ell)z^2+(n-1-\ell)\Big).
\end{align*}

As a consequence of Theorem~\ref{thm:F_L^p}, one obtains the following characterization (see \cite[Cor.~2.3]{Lauret-multip}).

\begin{theorem}\label{thm:charact[0,p]}
Let $0\leq p_0\leq n-1$ and let $L$ and $L'$ be $(2n-1)$-dimensional lens orbifolds with associated congruence lattices $\mathcal L$ and $\mathcal L'$ respectively.
Then, $L$ and $L'$ are $p$-isospectral for every $0\leq p\leq p_0$ if and only if
\begin{equation}\label{eq1:condition[0,p_0]-isosp}
\sum_{\ell=0}^n \ell^h \, \vartheta_{\mathcal L}^{(\ell)}(z) =
\sum_{\ell=0}^n \ell^h \, \vartheta_{\mathcal L_{\Gamma'}}^{(\ell)}(z)
\qquad\text{for all }0\leq h\leq p_0.
\end{equation}
In particular, $L$ and $L'$ are $0$-isospectral if and only if $\vartheta_{\mathcal L}(z)= \vartheta_{\mathcal L'}(z)$.
Moreover, $L$ and $L'$ are $p$-isospectral for all $p$ if and only if $\vartheta_{\mathcal L}^{(\ell)}(z)=\vartheta_{\mathcal L'}^{(\ell)}(z)$ for all $0\leq \ell\leq n$.
\end{theorem}

\subsection{Rational form for the one-norm generating functions}

For $q$ a positive integer and $\mathcal L\subset P(G)$, we define
\begin{align}
C(q) &=\{\mu=\textstyle\sum_{i}a_i\varepsilon_i\in P(G): |a_i|<q\quad\forall\,i\},\label{eq1:reduceterms} \\
N_{\mathcal L}^{\textrm{red}}(k,\ell) &= \#\{\mu\in C(q)\cap\mathcal L: \norma{\mu}=k,\;Z(\mu)=\ell\},\notag \\
\Phi_{\mathcal L}^{(\ell)}(z) &= \sum_{k\geq0} N_{\mathcal L}^{\textrm{red}}(k,\ell) z^k.\notag
\end{align}
Note that $\Phi_{\mathcal L}^{(\ell)}(z)$ is a polynomial of degree at most $(n-\ell)(q-1)$.

\begin{theorem}\label{thm:theta^l-rational}
Let $L$ be a $(2n-1)$-dimensional lens orbifold with fundamental group of order $q$ and associated congruence lattice $\mathcal L$.
Then, for each $0\leq \ell\leq n$,
\begin{equation*}
\vartheta_{\mathcal L}^{(\ell)} (z)
=\frac{1}{(1-z^{q})^{n-\ell}}\sum_{s=0}^{n-\ell} 2^s\binom{\ell+s}{s} z^{sq} \,\Phi_{\mathcal L}^{(\ell+s)}(z).
\end{equation*}
Moreover,
\begin{equation*}
\vartheta_{\mathcal L}(z)
=\frac{1}{(1-z^{q})^{n}} \, \displaystyle\sum_{t=0}^{n} z^{tq}\sum_{\ell=t}^n \binom{\ell}{t}  \Phi_{\mathcal L}^{(\ell)}(z).
\end{equation*}
\end{theorem}

\begin{lemma}\label{lem:theta^n-1}
Let $L=L(q;s)$ be a $(2n-1)$-dimensional lens space and let $\mathcal L$ be its associated congruence lattice.
Then $\vartheta_{\mathcal L}^{(n)}(z)=1$ and
$$
\vartheta_{\mathcal L}^{(n-1)}(z)=  \frac{2nz^q}{1-z^q}.
$$
\end{lemma}

\begin{proof}
Theorem~\ref{thm:theta^l-rational} ensures that $\vartheta_{\mathcal L}^{(n)}(z) = \Phi_{\mathcal L}^{(n)}(z)$ and $\vartheta_{\mathcal L}^{(n-1)}(z) = \Phi_{\mathcal L}^{(n-1)}(z) + 2nz^q\Phi_{\mathcal L}^{(n)}(z)$.
The assertions follow from $\Phi_{\mathcal L}^{(n)}(z)=1$ and $\Phi_{\mathcal L}^{(n-1)}(z)=0$.
Indeed, the first one is clear since the only element in $\Z^n$ with $n$ zero coordinates is the zero vector, which is always in $\mathcal L$.
The second one holds because a vector $\mu$ with $n-1$ zero coordinates has the form $ke_i$ for some $k\neq0$ and $1\leq i \leq n$, then $\mu$ is in $\mathcal L=\mathcal L(q;s_1,\dots,s_n)$ if and only if $ks_i\equiv 0 \pmod q$, which is equivalent to $k\equiv 0\pmod q$ because $\gcd(q,s_j)=1$ for all $j$.
\end{proof}

\section{Computational results}\label{sec:data}

In this section we include the consequences of the computational results.
This article is accompanied with the appendix \cite{Lauret-appendix}, which includes the computational results and the algorithms used.
More precisely, \cite{Lauret-appendix} shows the list of all $I$-isospectral families among $(2n-1)$-dimensional lens orbifolds for $2n-1=5$, $7$, $9$, $11$, $13$, $15$, $17$ with fundamental group of order $q\leq 200$, $120$, $75$, $59$, $44$, $32$, $27$ respectively.

Here, in each (odd) dimension, we will analyze the subsets $I$ for which there exist $I$-isospectralities among lens orbifolds with fundamental groups of low orders.
More precisely, for fixed $n$ satisfying $3\leq n\leq 6$ and each subset $I$ of $\{0,1,\dots,n-1\}$, Tables~\ref{table:summarydim5}--\ref{table:summarydim11} will indicate if there is a pair of $(2n-1)$-dimensional non-isometric $I$-isospectral lens spaces (or lens orbifolds) with fundamental group of low order.
The analogous tables in dimensions $13$, $15$ and $17$ are included in \cite{Lauret-appendix}.

When no example of $I$-isospectral pairs occurs, we will give a justification in case we are aware of it.
The most common justification for the non-existence of an $I$-isospectral pair is the `hole obstruction' in Proposition~\ref{prop:hole-obstruction}.
For instance, it is the case for every $I$ containing $1$ but not $0$, or containing $0$ and $2$ but not $1$.

The lowest dimension considered is $5$ since no examples appear in the computations for dimension $3$. 
Indeed, Ikeda and Yamamoto~\cite{IkedaYamamoto79}\cite{Yamamoto80} proved the non-existence of non-isometric $0$-isospectral $3$-dimensional lens spaces. 
Of couse, $1$-isospectrality is also impeded by Proposition~\ref{prop:hole-obstruction}. 
Moreover, recently Shams and Hunsicker~\cite{ShamsHunsicker17} extend this result to $3$-dimensional lens orbifolds.

\subsection{Dimension 5}
The computational results in dimension 5 included in \cite{Lauret-appendix} are summarized in Table~\ref{table:summarydim5}.
In this case, the hole obstruction (Proposition~\ref{prop:hole-obstruction}) is an impediment to the existence of $I$-isospectral lens orbifolds for $I=\{0,2\},\{1\},\{1,2\}$.

\begin{table}
\caption{Dimension 5}\label{table:summarydim5}
\begin{tabular}{|c|c|c|} \hline
\raisebox{-5pt}{\rule{0mm}{18pt}}&Lens spaces & Lens orbifolds\\ \hline \hline
    & \multicolumn{2}{c|}{$\emptyset$, $\{0\}$, $\{0,1,2\}$.} \\ \cline{2-3}
$\exists$
    &
    &$\{2\}$.\\ \hline
$\nexists$ Prop.~\ref{prop:hole-obstruction}
    & \multicolumn{2}{c|}{$\{0,2\}$, $\{1\}$, $\{1,2\}$.}\\ \hline
$\nexists$ Thm.~\ref{thm:[0,n-2]-isosp}
    & $\{0,1\}$.& \\ \hline
$\nexists$ Rem.~\ref{rem:[0,n-2]-isosporbifolds}
    & &$\{0,1\}$. \\ \hline
$\nexists$ Thm.~\ref{thm:n-2elements}
    &$\{2\}$. & \\ \hline
\end{tabular}
\end{table}

The existence of $I$-isospectral $5$-dimensional lens spaces for $I=\{0\}$ was well known.
Actually, Theorem~\ref{thm:mainIkeda} for $q=11$ (the smallest non-trivial example) gives a pair of $5$-dimensional $\{0\}$-isospectral lens spaces.
Furthermore, there are $\{0\}$-isospectral families of lens orbifolds which are not lens spaces.
An example with smallest fundamental group is given by the pair $$\{L(15;1,2,6), L(15;1,3,4)\}.$$
The classification of $5$-dimensional $\{0\}$-isospectral lens spaces or lens orbifolds seems to be a difficult task.

Recently, the existence of pairs of non-isometric $5$-dimensional $\{0,1,2\}$-isospectral lens spaces was shown.
This implies that the lens spaces are $p$-isospectral for all $p$ since the $p$-spectrum of a $5$-dimensional lens space coincides with its $(5-p)$-spectrum.
Many such examples have been found in \cite{LMR-onenorm}.
It would be interesting to have a classification of $5$-dimensional lens spaces that are $p$-isospectral for all $p$.
The computational results give evidences that such families come only in pairs. 
DeFord and Doyle~\cite{DeFordDoyle14} have made an important step on this problem in arbitrary dimension.
It is very interesting that all such examples are given by lens spaces, that is, there is no any example of non-isometric $5$-dimensional `pure' lens orbifolds $p$-isospectral for all $p$.

Theorem~\ref{thm:[0,n-2]-isosp} ensures the non-existence of $5$-dimensional $\{0,1\}$-isospectral lens spaces.
The data gives evidences that the same is true for $5$-dimensional lens orbifolds (see Remark~\ref{rem:[0,n-2]-isosporbifolds}).

The non-existence of $5$-dimensional $\{2\}$-isospectral lens spaces is proved in Theorem~\ref{thm:n-2elements}.
It was surprising the existence of $\{2\}$-isospectral lens orbifolds.
Conjecture~\ref{conj:5-dim[2]-isosp} claims its classification.

A very interesting fact is that all pairs of $5$-dimensional $\{2\}$-isospectral lens orbifolds have different isotropy types.
For example, $\{L(8;0,1,3),L(8;1,3,4)\}$ is the pair with smallest fundamental group in this situation.
On the one hand, the maximal isotropy group in $L(8;1,3,4)$ is given in the class of the point $(0,0,1)\in\C^3$ (see Remark~\ref{rem:lensinC^n} for notation).
Such group is isomorphic to the cyclic group of order $4$.
On the other hand, the class of the element $(1,0,0)$ in $L(8;0,1,3)$ has isotropy group isomorphic to the cyclic group of order $8$, which of course has maximal order.

\subsection{Dimension 7}

\begin{table}
\caption{Dimension 7}\label{table:summarydim7}
\begin{tabular}{|c|c|c|} \hline
\raisebox{-5pt}{\rule{0mm}{18pt}}&Lens spaces & Lens orbifolds\\ \hline \hline
    & \multicolumn{2}{c|}{$\emptyset$, $\{0\}$, $\{0,1\}$, $\{0,1,2,3\}$.} \\ \cline{2-3}
$\exists$
    &
    & $\{2\}$, $\{3\}$.  \\ \hline
$\nexists$ Prop.~\ref{prop:hole-obstruction}
    & \multicolumn{2}{c|}{$\{0,1,3\}$, $\{0,2\}$, $\{0,2,3\}$, $\{1\}$, }\\
(hole obstruction)
    &\multicolumn{2}{c|}{$\{1,2\}$, $\{1,2,3\}$, $\{1,3\}$. }
\\ \hline
$\nexists$ Thm.~\ref{thm:[0,n-2]-isosp}
    &$\{0,1,2\}$.& \\ \hline
$\nexists$ Rem.~\ref{rem:[0,n-2]-isosporbifolds}
    &&$\{0,1,2\}$. \\ \hline
$\nexists$ Thm.~\ref{thm:n-2elements}
    & $\{0,3\}$, $\{2,3\}$. &\\ \hline
$\nexists$ Rem.~\ref{rem:n-2elementsorbifolds}
    && $\{0,3\}$.\\ \hline
$\nexists$ ?
    & $\{2\}$, $\{3\}$.
    & $\{2,3\}$. \\ \hline
\end{tabular}
\end{table}

Table~\ref{table:summarydim7} shows a summary of the computational results in dimension 7.
Similarly as in the previous dimension, Proposition~\ref{prop:hole-obstruction} ensures the non-existence of $I$-isospectral pairs for subsets $I$ with simple holes, namely, $\{0,1,3\}$, $\{0,2\}$, $\{0,2,3\}$, $\{1\}$, $\{1,2\}$, $\{1,2,3\}$, $\{1,3\}$.

The existence of $I$-isospectral $7$-dimensional lens spaces for $I=\{0\}$ and  $\{0,1\}$ were already known.
Indeed, Theorem~\ref{thm:mainIkeda} for $q=13$ shows examples to these cases.
Moreover, there exist such examples for pure lens orbifolds.
For instance, any pair of $5$-dimensional $0$-isospectral lens spaces induces a pair, via Proposition~\ref{prop:adding0s}, of $7$-dimensional $0$-isospectral lens orbifolds, which are not lens spaces.
For example, see in \cite[Table 2]{Lauret-appendix} the pair for $q=11$.

The smallest example (minimal $q$) of $\{0,1\}$-isospectral $7$-dimensional lens orbifolds which are not lens spaces is
$$
\{L(36;1,3,5,17),\ L(36;1,3,7,11)\}.
$$
An interesting fact is that $L(36;1,3,5,17)$ is isometric to $L(36;1,7,11,15)$ by multiplying its parameters by $5$ and making a reordering.
In other words, the $5$-dimensional lens spaces $L(36;1,5,17)$ and $L(36;1,7,11)$ are isometric, but to add the parameter $3$ to each of them produces a pair of $7$-dimensional non-isometric $\{0,1\}$-isospectral lens orbifolds.

As in the previous dimension, $q=49$ is the smallest $q$ with examples of pairs among $7$-dimensional lens spaces $p$-isospectral for all $p$ (i.e.\ $\{0,1,2,3\}$-isospectral).
In \cite{LMR-onenorm} and \cite{LMR-survey} are many more examples.
Differently to the previous dimension, there exist $7$-dimensional pure lens orbifolds (i.e.\ lens orbifolds with non-trivial maximal isotropy group) $p$-isospectral for all $p$.
Moreover, such examples can be constructed by adding certain non-coprime parameter to a pair of $5$-dimensional lens spaces that are $p$-isospectral for all $p$.
For instance, the pair $\{L(49;1,6,15), L(49;1,6,20)\}$ of $5$-dimensional lens spaces is $p$-isospectral for all $p$, and it induces the pairs
\begin{align*}
&
\begin{cases}
L(49;1,6,15, 0), \\ L(49;1,6,20, 0),
\end{cases}
&&
\begin{cases}
L(49;1,6,15, 7), \\ L(49;1,6,20, 7),
\end{cases}
&&
\begin{cases}
L(49;1,6,15,14), \\ L(49;1,6,20,14),
\end{cases}
&&
\begin{cases}
L(49;1,6,15,21), \\ L(49;1,6,20,21).
\end{cases}
\end{align*}
which are lens orbifolds $p$-isospectral for all $p$.

Theorem~\ref{thm:[0,n-2]-isosp} ensures the non-existence of $\{0,1,2\}$-isospectral lens spaces.
It is expected that this fact holds for lens orbifolds (see Remark~\ref{rem:[0,n-2]-isosporbifolds}).
We prove in Theorem~\ref{thm:n-2elements} the non-existence of pairs of $\{2,3\}$-isospectral lens spaces.
There are evidences to think on the non-existence of $\{2,3\}$-lens orbifolds, and $I$-isospectral lens spaces for $I=\{2\},\{3\}$.

There are many pairs of $\{2\}$ and $\{3\}$-isospectral pure lens orbifolds, beginning from small choices of $q$.
Actually, one can see a pattern for the values of $q$ satisfying this condition (see Conjecture~\ref{conj:7-dim[2][3]-isosp}).

\subsection{Dimension 9}

\begin{table}
\caption{Dimension 9}\label{table:summarydim9}
\begin{tabular}{|c|c|c|} \hline
\raisebox{-5pt}{\rule{0mm}{18pt}}\raisebox{-5pt}{\rule{0mm}{18pt}}&Lens spaces & Lens orbifolds\\ \hline \hline
    & \multicolumn{2}{c|}{$\emptyset$, $\{0\}$, $\{0,1\}$, $\{0,1,2\}$, $\{0,1,2,3,4\}$.} \\ \cline{2-3}
$\exists$
    &
    & $\{2,3\}$, $\{2,3,4\}$, $\{3,4\}$, $\{4\}$. \\ \hline
$\nexists$ Prop.~\ref{prop:hole-obstruction}
    &\multicolumn{2}{c|}{$\{0,1,2,4\}$, $\{0,1,3\}$, $\{0,1,3,4\}$, $\{0,2\}$, $\{0,2,3\}$, }\\
(hole
    &\multicolumn{2}{c|}{$\{0,2,3,4\}$, $\{0,2,4\}$, $\{1\}$, $\{1,2\}$, $\{1,2,3\}$, } \\
obstruction)
    &\multicolumn{2}{c|}{$\{1,2,3,4\}$, $\{1,2,4\}$, $\{1,3\}$, $\{1,3,4\}$, $\{1,4\}$, $\{2,4\}$.} \\ \hline
$\nexists$ Thm.~\ref{thm:[0,n-2]-isosp}
    &$\{0,1,2,3\}$.& \\ \hline
$\nexists$ Rem.~\ref{rem:[0,n-2]-isosporbifolds}
    &&$\{0,1,2,3\}$. \\ \hline
$\nexists$ Thm.~\ref{thm:n-2elements}
    &$\{0,1,4\}$, $\{0,3,4\}$, $\{2,3,4\}$. & \\ \hline
$\nexists$ Rem.~\ref{rem:n-2elementsorbifolds}
    &&$\{0,1,4\}$, $\{0,3,4\}$.\\ \hline
    &\multicolumn{2}{c|}{$\{0,3\}$, $\{0,4\}$, $\{2\}$, $\{3\}$.}\\ \cline{2-3}
$\nexists$ ?
    & $\{2,3\}$, $\{3,4\}$, $\{4\}$.
    & \\ \hline
\end{tabular}
\end{table}

The computational results in dimension $9$ for $q\leq 75$ are in \cite[Table~3]{Lauret-appendix}, while their summary is in Table~\ref{table:summarydim9}.
The hole obstruction explains the non-existence of $I$-isospectral lens orbifolds for $16$ cases among $32$.

Theorem~\ref{thm:mainIkeda} does not produce any example in dimension $9$.
However, there exist examples of $I$-isospectral pair of lens spaces for $I=\{0\}$, $\{0,1\}$, and $\{0,1,2\}$.
The corresponding minimal values of $q$ are $16$, $17$ and $27$ respectively.

The pair of $\{0,1,2,3,4\}$-isospectral lens spaces with smallest $q$ appears in $q=64$.
However, $q=49$ is the smallest $q$ such that there is a pair of $\{0,1,2,3,4\}$-isospectral lens orbifolds.
Furthermore, one can see that these examples can be constructed from those in lower dimensions, in a similar way as was explained in dimension $7$.

Theorem~\ref{thm:[0,n-2]-isosp} ensures the non-existence of $\{0,1,2,3\}$-isospectral lens spaces (see also Remark~\ref{rem:[0,n-2]-isosporbifolds} for lens orbifolds).
Theorem~\ref{thm:n-2elements} guarantees the non-existence of pairs of $I$-isospectral lens spaces for $I=\{0,1,4\}, \{0,3,4\}, \{2,3,4\}$ (see Remark~\ref{rem:n-2elementsorbifolds} for $I$-isospectral lens orbifolds for $I=\{0,1,4\}, \{0,3,4\}$).

We do not know a reason for the non-existence of the cases indicated in the last two rows in Table~\ref{table:summarydim9}.
Also, it is somewhat surprising to see the existence of $\{2,3,4\}$-isospectral lens orbifolds.

\subsection{Dimension 11}

The computational results in dimension $11$ for $q\leq 59$ are in \cite[Table~4]{Lauret-appendix}, while their summary can be found in Table~\ref{table:summarydim11}.
In this case, there are $36$ non-existence cases among $64$ which are explained by the hole obstruction.

\begin{table}
\caption{Dimension 11}\label{table:summarydim11}
\begin{tabular}{|c|c|c|} \hline
\raisebox{-5pt}{\rule{0mm}{18pt}}&Lens spaces & Lens orbifolds\\ \hline \hline
    & \multicolumn{2}{c|}{$\emptyset$, $\{0\}$, $\{0,1\}$, $\{0,1,2\}$, $\{0,1,2,3\}$, $\{0,1,2,3,4,5\}$, $\{4,5\}$.} \\ \cline{2-3}
$\exists$
    &
    & $\{0,3\}$, $\{0,5\}$, $\{2,3,4\}$,\\
    &
    & $\{3\}$, $\{3,4,5\}$, $\{4\}$, $\{5\}$.
\\ \hline
$\nexists$ Thm.~\ref{thm:[0,n-2]-isosp}
    &$\{0,1,2,3,4\}$.  &\\ \hline
$\nexists$ Rem.~\ref{rem:[0,n-2]-isosporbifolds}
    &&$\{0,1,2,3,4\}$.  \\ \hline
    & $\{0,1,2,5\}$, $\{0,1,4,5\}$,& \\
$\nexists$ Thm.~\ref{thm:n-2elements}
    & $\{0,3,4,5\}$, $\{2,3,4,5\}$.&
\\ \hline
    && $\{0,1,2,5\}$, $\{0,1,4,5\}$, \\
$\nexists$ Rem.~\ref{rem:n-2elementsorbifolds}
    &&$\{0,3,4,5\}$. \\ \hline
    & \multicolumn{2}{c|}{$\{0,4\}$, $\{0,1,4\}$, $\{0,1,5\}$, $\{0,3,4\}$, $\{0,4,5\}$, } \\
    & \multicolumn{2}{c|}{$\{2\}$, $\{2,3\}$, $\{2,5\}$, $\{3,4\}$. } \\
    \cline{2-3}
$\nexists$ ?
    & $\{0,3\}$, $\{0,5\}$, $\{2,3,4\}$,
    & $\{2,3,4,5\}$.   \\
    & $\{3\}$, $\{3,4,5\}$, $\{4\}$, $\{5\}$.
    &\\ \hline
\end{tabular}
\end{table}

Theorem~\ref{thm:mainIkeda} for $q=17$ gives examples of $11$-dimensional $I$-isospectral lens spaces for $I=\{0\}$, $\{0,1\}$ and $\{0,1,2\}$.
Furthermore, there exist pairs of $11$-dimensional non-isometric $\{0,1,2,3\}$-isospectral lens spaces.
The smallest value of $q$ for such example is $52$.

Although there is no pair of $11$-dimensional lens spaces that are $p$-isospectral for all $p$ (i.e.\ $\{0,1,2,3,4,5\}$-isospectral) in \cite[Table~4]{Lauret-appendix}, such examples do exist.
By computations made by the author related to the articles \cite{LMR-onenorm} and \cite{LMR-survey}, the examples with minimum value of $q$ is when $q=100$.
They are
\begin{equation*}
\begin{cases}
L(100;1,11,21,31,41,61),\\
L(100;1,11,21,31,41,81),
\end{cases}
\qquad
\begin{cases}
L(100;1,11,21,41,71,81),\\
L(100;1,11,21,41,51,81).
\end{cases}
\end{equation*}
One can check by using \cite[Thm.~1]{DeFordDoyle14} that they are $p$-isospectral for all $p$.

As it was expected, there are pairs of $11$-dimensional lens orbifolds $p$-isospectral for all $p$  with fundamental group of order $49$.
They can be constructed by using the previous examples of lens spaces $p$-isospectral for all $p$ of lower dimension.

Similarly as in the previous dimensions, Theorem~\ref{thm:[0,n-2]-isosp} and Theorem~\ref{thm:n-2elements} explains the non-existence of a couple of cases for $11$-dimensional lens spaces, and Remark~\ref{rem:[0,n-2]-isosporbifolds} and Remark~\ref{rem:n-2elementsorbifolds} say something about the lens orbifold case.

There are $15$ (resp.\ $10$) cases among $64$ for which we do not know the existence of $I$-isospectral lens spaces (resp.\ lens orbifolds).

\subsection{Conclusions and remarks}

In the previous tables, for a fixed dimension $2n-1$, there are five possibilities for each $I$ among the $2^n$ subsets of $\{0,\dots,n-1\}$, namely,
\begin{itemize}
\item there is an $I$-isospectral family in the computational results,
\item the existence of an $I$-isospectral family is obstructed by Proposition~\ref{prop:hole-obstruction},
\item the existence of an $I$-isospectral family is obstructed by Theorem~\ref{thm:[0,n-2]-isosp} or Remark~\ref{rem:[0,n-2]-isosporbifolds},
\item the existence of an $I$-isospectral family is obstructed by Theorem~\ref{thm:n-2elements} or Remark~\ref{rem:n-2elementsorbifolds},
\item there is no any $I$-isospectral family in the computational results and this cannot be explained by the above reasons.
\end{itemize}
Table~\ref{table:numbersubsets} shows how the number of subsets $I$ in each possible situation changes according to the growth of dimension up to $17$.

We recall that in Tables~\ref{table:summarydim5}--\ref{table:summarydim11}, a subset $I$ in the row `$\nexists$ ?' means that the computational results show no $I$-isospectral family which cannot be explained by the reasons above. 
Probably there exists such a family with fundamental group of higher order than the considered by the computer.
The author believes that there should exist, in dimensions $13$, $15$ and $17$, more subsets $I$ satisfying that there is an $I$-isospectral pair.

\begin{table}
\caption{It is shown the number of subsets $I$ of $\{0,\dots,n-1\}$ such that there is no $(2n-1)$-dimensional $I$-isospectral family of lens spaces and orbifolds in the computational results.
In a fixed dimension, the left (resp.\ right) column considers lens spaces (resp.\ lens orbifolds).
The third row indicates the cases explained in Section~\ref{sec:evidencedfacts} (i.e.\ Thm.~\ref{thm:[0,n-2]-isosp}, Rem.~\ref{rem:[0,n-2]-isosporbifolds}, Thm~\ref{thm:n-2elements} or Rem.~\ref{rem:n-2elementsorbifolds}).
}\label{table:numbersubsets}
\begin{center}
\begin{tabular}{c|cc|cc|cc|cc|cc|cc|cc|}
existence or reason &\multicolumn{14}{c|}{dimension}\\
for the non-existence
&\multicolumn{2}{c|}{5} &\multicolumn{2}{c|}{7} &\multicolumn{2}{c|}{9} &\multicolumn{2}{c|}{11} &\multicolumn{2}{c|}{13} &\multicolumn{2}{c|}{15} &\multicolumn{2}{c|}{17} \\ \hline\hline
$\exists$
    &3&4& 4&6& 5&9& 7&14& 9&18& 6&23& 7&27\\ \hline
$\nexists$ Prop.~\ref{prop:hole-obstruction}
    &\multicolumn{2}{c|}{3} & \multicolumn{2}{c|}{7}& \multicolumn{2}{c|}{16} &\multicolumn{2}{c|}{36} &\multicolumn{2}{c|}{79} &\multicolumn{2}{c|}{170} &\multicolumn{2}{c|}{361} \\ \hline
$\nexists$ Sect.~\ref{sec:evidencedfacts}
    &2&1& 3&2& 4&3& 5&4& 6&5& 7&6& 8&7\\ \hline
$\nexists$ ?
    &0&0& 2&1& 7&4& 16&10& 34&26& 73&57& 136&117\\ \hline \hline
Total & \multicolumn{2}{c|}{8}& \multicolumn{2}{c|}{16}& \multicolumn{2}{c|}{32} &\multicolumn{2}{c|}{64}& \multicolumn{2}{c|}{128}& \multicolumn{2}{c|}{256}& \multicolumn{2}{c|}{512}
\end{tabular}
\end{center}
\end{table}

We end this section with other observations.

\begin{remark}\label{rem:isotropytype}
At the end of the subsection corresponding to dimension $5$, we observed that the singular points in the lens orbifolds of any $\{2\}$-isospectral pair have different isotropy type.
In higher dimensions, a similar situation is repeated many times for $I$-isospectral families with $0\not\in I$.
More precisely, when $L(q;s_1,\dots,s_n)$ and $L(q;s_1',\dots,s_n')$ are $p$-isospectral for some $0<p\leq n-1$ and not $0$-isospectral, we usually have that $$
\{\!\{\gcd(q,s_j):1\leq j\leq n\}\!\} \neq \{\!\{\gcd(q,s_j'):1\leq j\leq n\}\!\}.
$$
However, this is not always the case.
For instance, the family
\begin{align*}
\{L(27;1, 3, 8, 10),\; L(27;1, 6, 8, 10),\; L(27;1, 8, 10, 12)\}
\end{align*}
is $\{2\}$-isospectral and $\{L(12;1, 1, 1, 3, 5),\; L(12;1, 1, 3, 5, 5)\}$ is $\{4\}$-isospectral.

On the other hand, it is expected that $0$-isospectral lens orbifolds have the same isotropy types (see Remark~\ref{rem:L(q/d;s)orbifolds}).
\end{remark}

\begin{remark}\label{rem:p-isospp>0}
Gornet and McGowan were interested on the existence of lens spaces $p_0$-isospec\-tral for some $p_0>0$ but not $p$-isospectral for all $0\leq p<p_0$.
We can see in Table~\ref{table:summarydim5} that such example does not exist in dimension $5$.
Indeed, in this case, $1$-isospectrality implies $0$-isospectrality by Proposition~\ref{prop:hole-obstruction} and $2$-isospectrality implies $p$-isospectrality for all $p$ by Theorem~\ref{thm:n-2elements}.
Also, Table~\ref{table:summarydim7} says that we do not know if such examples can exist in dimension $7$ since the problem of the existence of an $I$-isospectral pair is open for $I=\{2\}$ and $\{3\}$.
A similar situation happens in dimension $9$ according to Table~\ref{table:summarydim9}.
In dimension $11$ there exists a pair of $\{4,5\}$-isospectral lens spaces, in particular they are $4$-isospectral but not $p$-isospectral for all $0\leq p\leq 3$.
In \cite{Lauret-appendix} there are examples in dimension $13$ of $\{4\}$-isospectrality and $\{6\}$-isospectrality among lens spaces.

The same existence problem considering lens orbifolds is quite different since it is true from dimension $5$.
Moreover, there are many more examples of $I$-isospectral pure lens orbifolds with $0\not\in I$ than for lens spaces.
\end{remark}

\section{Evidenced facts}\label{sec:evidencedfacts}

In this section we list many facts observed in the computational results summarized in the previous section.
We will prove many of them by using the tools from Section~\ref{sec:preliminaries}.
We will also conjecture some other facts evidenced by the computational results.

\subsection{New examples increasing the dimension} \label{subsec:increasing-dimension}

We easily see from the data that one can add (or delete) $0$'s to the parameters of $0$-isospectral lens orbifolds and obtains $0$-isospectral lens orbifolds of higher (lower) dimension.

\begin{proposition}\label{prop:adding0s}
Let $L_1 =L(q;s_1,\dots,s_n)$, $L_2=L(q;s_1,\dots,s_n,0,\dots,0)$, $L_1'=L(q;s_1',\dots,s_n')$, and $L_2'=L(q;s_1',\dots,s_n',0,\dots,0)$,
where the parameter $0$ is repeated $m$ times in $L_2$ and $L_2'$.
Then, $L_1$ and $L_1'$ are $0$-isospectral if and only if $L_2$ and $L_2'$ are $0$-isospectral.
\end{proposition}

\begin{proof}
By \eqref{eq:F_L^0abreviado}, we have that
\begin{align*}
\widetilde F_{L_2}^0(z)
&= \sum_{h=0}^{q-1} \left(\prod_{j=1}^n\frac{1}{(z-\xi_q^{hs_j})(z-\xi_q^{-hs_j})}\right) \frac{1}{(z-1)^{2m}}= \frac{1}{(z-1)^{2m}} \widetilde F_{L_1}^0(z),
\end{align*}
and similarly for $L_1'$ and $L_2'$.
It follows immediately that $\widetilde F_{L_1}^0(z)=\widetilde F_{L_1'}^0(z)$ if and only if $\widetilde F_{L_2}^0(z)=\widetilde F_{L_2'}^0(z)$, and the proof is complete by Remark~\ref{rem:F_L^0}.
\end{proof}

Note that, the resulting spaces after adding $0$'s are always lens orbifolds with singular points, that is, they cannot be lens spaces.
For instance, $L(11;1,2,3)$ and $L(11;1,2,4)$ are $5$-dimensional non-isometric lens spaces and $L(11;1,2,3,0,\dots,0)$ and $L(11;1,2,4,0\dots,0)$, where $m$ is the number of zero coordinates added to each of them, are $(5+2m)$-dimensional non-isometric $0$-isospectral lens orbifolds.

\begin{remark}
We note that to add zero parameters to $p$-isospectral pairs with $p>0$ does not return $p$-isospectral pairs.
For example, $L(8;0,1,3)$ and $L(8;1,3,4)$ are $2$-isospectral, though $L(8;0,0,1,3)$ and $L(8;0,1,3,4)$ are not $2$-isospectral.
Moreover, there is no any $p$-isospectrality among $7$-dimensional lens spaces with fundamental group of order $q=8$.
\end{remark}

\subsection{New examples for increasing $q$}\label{subsec:increasing-q}

The next result gives a necessary condition to construct pairs of lens spaces covered by a pair of $0$-isospectral lens spaces.

\begin{theorem}\label{thm:L(q/d,s)}
Let $L=L(q;s)$ and $L=L(q;s')$ be $(2n-1)$-dimensional lens spaces.
If $L(q;s)$ y $L(q;s')$ are $0$-isospectral, then $L(q_1;s)$ y $L(q_1;s')$ are also $0$-isospectral for every positive divisor $q_1$ of $q$.
\end{theorem}

\begin{proof}
By \eqref{eq:F_L^0abreviado}, we have that
\begin{align}\label{eq:tildeF_L^0(z)-suma-divisores}
\widetilde F_{L}^0(z)
&= \sum_{d\mid q} \sum_{0\leq h<q\atop \gcd(h,q)=q/d} \prod_{j=1}^n\frac{1}{(z-\xi_q^{hs_j})(z-\xi_q^{-hs_j})} \\
&= \sum_{d\mid q} \sum_{h\in\Z_d^\times} \prod_{j=1}^n\frac{1}{(z-\xi_d^{hs_j})(z-\xi_d^{-hs_j})}.
\notag
\end{align}
Clearly, the poles of $\widetilde F_L^0(z)$ are primitive roots of unity of degree $d$ for some divisor $d$ of $q$.

We now assume that $L=L(q;s)$ and $L=L(q;s')$ are $0$-isospectral lens spaces, thus $\gcd(s_j,q)=\gcd(s_j',q)=1$ for all $1\leq j\leq n$.
Since $\widetilde F_{L}^0(z)= \widetilde F_{L'}^0(z)$, these functions have the same set of poles and, at each pole, they have the same order and singular part.
Hence, the sum of the singular parts of the poles at $d$-th primitive root of unity (i.e.\ $\xi_d^h$ for $h\in\Z_d^\times$) coincide, thus
$$
\sum_{h\in\Z_d^\times} \prod_{j=1}^n\frac{1}{(z-\xi_d^{hs_j})(z-\xi_d^{-hs_j})}
=
\sum_{h\in\Z_d^\times} \prod_{j=1}^n\frac{1}{(z-\xi_d^{hs_j'})(z-\xi_d^{-hs_j'})}
$$
for any divisor $d$ of $q$.
This implies that
$$
\widetilde F_{L(q_1,s)}^0(z) =
 \sum_{d\mid q_1} \sum_{h\in\Z_d^\times} \prod_{j=1}^n\frac{1}{(z-\xi_d^{hs_j})(z-\xi_d^{-hs_j})}
 = \widetilde F_{L(q_1,s')}^0(z) ,
$$
which completes the proof.
\end{proof}

Since lens spaces have cyclic fundamental group, their covering spaces are parameterized by the set of positive divisors of the order of the fundamental group.
More precisely, if $d$ is a positive divisor of $q$, then the $d$-covering space of $L(q;s_1,\dots,s_n)$  is $L(q/d;s_1,\dots,s_n)$.

Consequently, Theorem~\ref{thm:L(q/d,s)} says that if two lens spaces are $0$-isospectral, then their covering spaces of the same order are also $0$-isospectral.

\begin{remark}\label{rem:L(q/d;s)orbifolds}
The numerical evidences strongly support the claim that Theorem~\ref{thm:L(q/d,s)} is also valid for lens orbifolds.
However, the proof is not easy since the singular part of $\widetilde F_L^0(z)$ becomes unmanageable when $L$ is not a lens space.
Indeed, since $s_1,\dots,s_n$ are not necessarily coprime to $q$, it could be roots of unity of different orders in a same term in \eqref{eq:tildeF_L^0(z)-suma-divisores}.

If Theorem~\ref{thm:L(q/d,s)} were valid for lens orbifolds, then $0$-isospectral lens orbifolds $L(q;s_1,\dots,s_n)$ and $L(q;s_1',\dots,s_n')$ would satisfy
\begin{equation}\label{eq:conditionisotropy}
\{\!\{\gcd(q,s_j):1\leq j\leq n\}\!\} = \{\!\{\gcd(q,s_j'):1\leq j\leq n\}\!\}.
\end{equation}
These multisets (sets with multiplicities) are very important since they determine the isotropy types of the singular points of the corresponding lens orbifolds (see Subsection~\ref{subsec:lensspaces}).
\end{remark}

\subsection{$I$-isospectrality for $I$ having $n-1$ elements}\label{subsec:[0,n-2]-isosp}

One can observe from the data that there is no pair of $(2n-1)$-dimensional non-isometric lens orbifolds that are $I$-isospectral, if $I$ has $n-1$ elements. Indeed, the hole obstruction (Proposition~\ref{prop:hole-obstruction}) proves this fact for all cases except $I=\{0,\dots,n-2\}$ (i.e.\ $p$-isospectral for all $0\leq p\leq n-2$ but not $(n-1)$-isospectral).
The next result ensures the non-existence of them for lens spaces.

\begin{theorem}\label{thm:[0,n-2]-isosp}
Let $L=L(q;s)$ and $L=L(q;s')$ be $(2n-1)$-dimensional lens spaces.
If $L$ and $L'$ are $p$-isospectral for all $p$ satisfying $0\leq p\leq n-2$, then $L$ and $L'$ are also $(n-1)$-isospectral.
Consequently, there is no pair of non-isometric $(2n-1)$-dimensional $\{0,\dots,n-2\}$-isospectral lens spaces. 
\end{theorem}

\begin{proof}
Write $\mathcal L=\mathcal L(q;s)$ and $\mathcal L'=\mathcal L(q;s')$ as in \eqref{eq1:conglattice}.
In order to show that $L$ and $L'$ are $p$-isospectral for all $p$, by Theorem~\ref{thm:charact[0,p]}, we will prove that $\mathcal L$ and $\mathcal L'$ are $\norma{\cdot}^*$-isospectral, that is, $\vartheta_{\mathcal L}^{(\ell)}(z) = \vartheta_{\mathcal L'}^{(\ell)}(z)$ for all $0\leq \ell\leq n$.

Abbreviating $\vartheta^{(\ell)}(z)=\vartheta_{\mathcal L}^{(\ell)}(z)-\vartheta_{\mathcal L'}^{(\ell)}(z)$, Theorem~\ref{thm:charact[0,p]} implies that our hypothesis is equivalent to
\begin{equation}\label{eq:n-2equations}
0=\sum_{\ell=0}^{n} \ell^h\vartheta^{(\ell)}(z)\qquad\text{ for all }\;0\leq h\leq n-2.
\end{equation}
Thus, we have $n-1$ equations (one for each $h$) and $n+1$ variables, namely, $\vartheta^{(0)}(z)$,$\dots$,$\vartheta^{(n)}(z)$.
This holds for arbitrary lens orbifolds.

Since $L$ and $L'$ are lens spaces, then $\vartheta^{(n)}(z)=\vartheta^{(n-1)}(z)=0$ by Lemma~\ref{lem:theta^n-1}.
We conclude that \eqref{eq:n-2equations} has $n-1$ equations and $n-1$ variables.
Moreover, the coefficients form a Vandermonde matrix, which is non-singular.
Hence, $\vartheta^{(\ell)}(z)=0$ for all $\ell$, and the proof is complete.
\end{proof}

\begin{remark}\label{rem:[0,n-2]-isosporbifolds}
We conjecture that Theorem~\ref{thm:[0,n-2]-isosp} also holds for lens orbifolds.
To complete the proof, it only remains to show that $\vartheta_{\mathcal L}^{(n-1)}(z)=\vartheta_{\mathcal L'}^{(n-1)}(z)$, which is true if $L$ and $L'$ satisfy \eqref{eq:conditionisotropy}.
Indeed, $\mu$ is a vector in $\mathcal L$ with exactly $n-1$ zero coordinates if and only if $\mu=k\varepsilon_i$ for some non-zero integer $k$ such that $ks_i\equiv0\pmod q$.
Since the last condition is equivalent to $\gcd(q,s_i)$ divides $k$, then
\begin{equation}\label{eq:theta^n-1}
\vartheta_{\mathcal L}^{(n-1)}(z) =
\sum_{i=1}^n \frac{z^{\frac{\gcd(q,s_i)}q}}{1-z^{\frac{\gcd(q,s_i)}q}}.
\end{equation}
Hence, $\vartheta_{\mathcal L}^{(n-1)}(z)$ only depends on the multiset $\{\!\{\gcd(q,s_j):1\leq j\leq n\}\!\}$, which determines the isotropy types of $L$ (see Subsection~\ref{subsec:lensspaces}).
We conclude by Remark~\ref{rem:L(q/d;s)orbifolds} that the non-existence of non-isometric $\{0,\dots,n-2\}$-isospectral lens orbifolds follows by showing that Theorem~\ref{thm:L(q/d,s)} is valid for lens orbifolds.
\end{remark}

\subsection{$I$-isospectrality for $I$ having $n-2$ elements}\label{subsec:n-2elements}
Let
\begin{equation}
\mathcal A=
\begin{pmatrix}
A_1^{(0)}(z) &\dots& A_1^{(n-2)}(z) \\
\vdots & \ddots & \vdots \\
A_{n}^{(0)}(z) &\dots& A_{n}^{(n-2)}(z)
\end{pmatrix},
\end{equation}
with $A_{p}^{(\ell)}(z)$ as in \eqref{eq3:A_pl}.
The matrix $\mathcal A$ has size $n\times(n-1)$ and coefficients in $\C(z)$.

\begin{lemma}\label{lem:matrix}
If $n\leq 9$, then the square matrix obtained by deleting any row to $\mathcal A$ is non-singular.
\end{lemma}

\begin{proof}
We have checked this claim by using \cite{Sage}.
In the appendix \cite{Lauret-appendix}, the (non-zero) determinants of such matrices are indicated.
\end{proof}

\begin{remark}
The author believes that Lemma~\ref{lem:matrix} holds for every $n$.
However, the proof does not seem easy.
\end{remark}

\begin{theorem}\label{thm:n-2elements}
Let $I\subset \{0,\dots,n-1\}$ with $n-2$ elements, excluding the case $\{0,\dots,n-3\}$.
Then, there is no any $I$-isospectral pair of non-isometric $(2n-1)$-dimensional lens spaces.
\end{theorem}

\begin{proof}
By Proposition~\ref{prop:p-isospIkeda}, two $(2n-1)$-dimensional lens spaces $L$ and $L'$ are $p$-isospectral for all $p\in I$ if $F_L^p(z)=F_{L'}^p(z)$ for all $p$ satisfying $p\in I$ or $p-1\in I$.
Convince yourself that these are at least $n-1$ different equations, as a consequence of how we choose the subset $I$.
Of course, $F_L^{-1}(z)=F_{L'}^{-1}(z)$ is not considered as an equation since they are zero by convention.

By Theorem~\ref{thm:F_L^p}, each equation $F_L^p(z)=F_{L'}^p(z)$ for some $0\leq p\leq n-1$ is equivalent to
\begin{equation}
0= \sum_{\ell=0}^n A_p^{(\ell)}(z) \vartheta^{(\ell)}(z).
\end{equation}
Here, we are abbreviating $\vartheta^{(\ell)}(z)=\vartheta_{\mathcal L'}^{(\ell)}(z)-\vartheta_{\mathcal L'}^{(\ell)}(z)$.

We conclude that we have at least $n-1$ equations with $n-1$ variables $\vartheta^{(0)}(z),\dots, \vartheta^{(n-2)}(z)$.
Indeed, $\vartheta^{(n-1)}(z)=\vartheta^{(n)}(z)=0$ by Lemma~\ref{lem:theta^n-1} since $L$ and $L'$ are lens spaces.
Hence, we obtain that $\vartheta^{(\ell)}(z)=0$ for all $\ell$ since the associated matrix has rank $n-1$ by Lemma~\ref{lem:matrix}.
Consequently, $L$ and $L'$ are $p$-isospectral for all $p$ by Theorem~\ref{thm:charact[0,p]}, and the proof is complete.
\end{proof}

\begin{remark}
Theorem~\ref{thm:n-2elements} cannot be extended to $I=\{0,\dots,n-3\}$ since there are $\{0,\dots,n-3\}$-isospectral pairs of $(2n-1)$-dimensional lens spaces, for low values of $n$ (e.g. $n=3,4,5,6,7,$).
The author thinks that such examples does not exist in every dimension.
\end{remark}

\begin{remark}\label{rem:n-2elementsorbifolds}
In case Theorem~\ref{thm:L(q/d,s)} is true for lens orbifolds (see Remark~\ref{rem:L(q/d;s)orbifolds}), we obtain the following result:
if $I\subset\{0,\dots,n-1\}$ has $n-2$ elements, $I\neq \{0,\dots,n-3\}$ and $0\in I$, then there is no any pair of $I$-isospectral non-isometric $(2n-1)$-dimensional lens orbifolds.
Indeed, following the proof of Theorem~\ref{thm:n-2elements}, it only remains to prove that $\vartheta_{\mathcal L'}^{(n-1)}(z)=\vartheta_{\mathcal L'}^{(n-1)}(z)$, which holds by Remark~\ref{rem:L(q/d;s)orbifolds} since $L$ and $L'$ are $0$-isospectral.

The assumption $0\in I$ above is essential, since Table~\ref{table:summarydim9} ensures the existence of pairs of $\{2,3,4\}$-isospectral $9$-dimensional lens orbifolds.
\end{remark}

\subsection{Orbifolds and manifolds}\label{subsec:Orbifoldsmanifolds}

We recall that a lens orbifold $L:=L(q;s_1,\dots,s_n)$ has a manifold structure if  $\gcd(q,s_j)=1$ for every $j$.
In this case, $L$ is called a lens space.

The next result follows immediately from \cite[Prop.~3.4(ii)]{GordonRossetti03}, since lens orbifolds of the same dimension share the Riemannian universal cover.

\begin{theorem}\label{thm:orbifoldmanifolds}
	Let $L$ and $L'$ be $0$-isospectral lens orbifolds.
	If $L$ is a lens space (i.e.\ a manifold), then $L'$ is also a lens space.
\end{theorem}

For arbitrary orbifolds, the result does not hold for $p>0$ since Gordon and Rossetti~\cite{GordonRossetti03} show a counterexample by using the middle degree, that is, there is a $2d$-dimensional manifold $d$-isospectral to a $2d$-dimensional orbifold with singularities.
However, in the special case of lens orbifolds, the computational results give strong evidences of the following claim.

\begin{conjecture}\label{conj:orbifold/manifold}
	If two $(2n-1)$-dimensional lens orbifolds $L$ and $L'$ are $p$-isospectral for some $0< p\leq n-1$ and $L$ is a lens space (i.e.\ a manifold), then $L'$ is a lens space.
\end{conjecture}

\subsection{$\{p\}$-isospectrality for $p>0$}\label{subsec:p-isosp}
Here, we investigate the examples appeared of $\{p\}$-isospec\-tral lens orbifolds with $p>0$ in dimension $5$ and $7$.
The computational results in \cite[Table 1]{Lauret-appendix} give evidences for the following claim.

\begin{conjecture}\label{conj:5-dim[2]-isosp}
The set of families of $5$-dimensional non-isometric $\{2\}$-isospectral lens orbifolds are given by pairs of the form
\begin{equation*}
\begin{cases}
  L(8t;4,t,3t)\\
  L(8t;8,t,3t)
\end{cases}
\qquad \text{for $t\geq1$ odd}.
\end{equation*}
\end{conjecture}

\begin{remark}
We sketch a possible proof.
If $L$ and $L'$ are $5$-dimensional $\{2\}$-isospectral lens orbifolds, then $F_{\Gamma}^{1}(z)=F_{\Gamma'}^{1}(z)$ and $F_{\Gamma}^{2}(z)=F_{\Gamma'}^{2}(z)$.
Theorem~\ref{thm:F_L^p} implies that these two equations are equivalent to the linear equation system
$$
\begin{array}{rrrcl}
(2z^2+2) \vartheta^{(0)}(z)
+&(3z^2+1) \vartheta^{(1)} (z)
+&4z^2\vartheta^{(2)} (z)
&=&0,\\
(z^4+4z^2+1) \vartheta^{(0)}(z)
+&(2z^4+2z^2) \vartheta^{(1)} (z)
+&(4z^4+2z^2) \vartheta^{(2)} (z)
&=&0,
\end{array}
$$
on the variables $\vartheta^{(\ell)} (z) := \vartheta_{\mathcal L}^{(\ell)}(z) - \vartheta_{\mathcal L'}^{(\ell)}(z)$ for $0\leq \ell\leq 2$.
By reducing this system we get that
$$
\begin{array}{rrcl}
(z^4+1) \vartheta^{(1)}(z)
+&4z^4\vartheta^{(2)}(z)
&=&0,\\
(2z^2-1) \vartheta^{(1)}(z)
+&2z^2\vartheta^{(0)}(z)
&=&0.
\end{array}
$$

This reduced system should be very useful to check that the sequence of pairs in Conjecture~\ref{conj:5-dim[2]-isosp} are $2$-isospectral.
A simple calculation shows that these pairs cannot be $p$-isospectral for $p=0,1$.

A more difficult task is to prove that this sequence exhausts such examples.
We could show that that $8$ divides to $q$ as follows.
We known that $\vartheta^{(\ell)}(z)$ is a rational function with denominator $(1-z^q)^{3-\ell}$ by Theorem~\ref{thm:theta^l-rational}.
We assume that $8$ does not divide to $q$, thus $\vartheta^{(1)}(z)$ converges at any primitive root of unity of order $8$, namely, $\xi_8^m$ with $m\in\Z$ odd.
Evaluating the first row of the previous system at $z=\xi_8^m$, we obtain that $0=4\xi_8^{4m}\vartheta^{(2)}(\xi_8^m)=-4\vartheta^{(2)}(\xi_8^m)$.
This equality should give a contradiction.
\end{remark}

Furthermore, \cite[Table 2]{Lauret-appendix} gives evidence on the following claim.

\begin{conjecture}\label{conj:7-dim[2][3]-isosp}
If $L(q;s)$ and $L(q;s')$ is a pair of $7$-dimensional non-isometric $\{2\}$-isospectral (resp.\ $\{3\}$-isospectral) lens orbifolds with fundamental group of order $q$, then $q$ is divisible by $3$ (resp.\ $5$).
\end{conjecture}

The higher dimensional cases are more complicated since the number of examples increase a lot.

\section{Conclusions}

Although lens spaces and lens orbifolds have provided several exotic isospectral examples, they are far away to answer in a positive way the question of R.~Miatello and J.P.~Rossetti~\cite[page 666]{MR-p-iso} (see also \cite[Problem~8.23]{CraioveanuPutaRassias-book}): 
\begin{quote}
	``Whether any possible combination of $p$-isospectrality is possible in a fixed dimension''.
\end{quote}
In other words, ``Is yes the answer of Question~\ref{question1} for every subset $I$?''.

According to Table~\ref{table:numbersubsets}, the major impediment is the well known hole obstruction (Proposition~\ref{prop:hole-obstruction}).  
Furthermore, there are some barriers from isolated results like Theorem~\ref{thm:[0,n-2]-isosp}, Theorem~\ref{thm:n-2elements}, and probably others explaining the not known cases. 

It would be desirable to have an explanation on the existence of $I$-isospectrality for each case $I$ in the rows `$\nexists$ ?'  in Tables~\ref{table:summarydim5}--\ref{table:summarydim11}.
More precisely, for each of those $I$, to show an $I$-isospectral pair, or to give a proof for its non-existence. 
Furthermore, it would be important to know if Theorem~\ref{thm:L(q/d,s)} holds also for lens orbifolds. 
According to Remark~\ref{rem:L(q/d;s)orbifolds}, this would show that the spectral information of a lens orbifold determines the isotropy type of its singular points. 
We note that this is not true for arbitrary orbifolds by \cite{ShamsStanhopeWebb06}.

The author hopes that this computational study  motivates someone to consider any related problem.

\bibliographystyle{plain}

\end{document}